\documentclass[12pt,a4paper,reqno]{amsart}
\usepackage{graphicx}
\usepackage{hyperref}
\usepackage[margin=3cm]{geometry}
\usepackage{mathrsfs}
\usepackage{overpic}
\usepackage[utf8]{inputenc}
\usepackage{amsmath}
\usepackage{amsfonts}
\usepackage{amssymb}
\usepackage{enumerate}
\usepackage{subcaption}
\usepackage{xcolor}
\usepackage{hyperref}

\newcommand{\R}{\ensuremath{\mathbb{R}}}

\renewcommand{\epsilon}{\varepsilon}
\renewcommand{\geq}{\geqslant}
\renewcommand{\leq}{\leqslant}

\newcommand*\Bell{\ensuremath{\boldsymbol\ell}}
\newcommand*\Balpha{\ensuremath{\boldsymbol\alpha}}
\newcommand*\Bbeta{\ensuremath{\boldsymbol\beta}}

\newtheorem{thm}{Theorem}[section]

\newtheorem{lem}[thm]{Lemma}
\newtheorem{prop}[thm]{Proposition}
\newtheorem{defi}[thm]{Definition}

\newtheorem{rem}[thm]{Remark}

\title[]
{Relativistic effects in the dynamics of a particle in a Coulomb field}

\author[R. Ortega \and D. Rojas]{Rafael Ortega$^1$ \and David Rojas$^2$}

\address{$^1$ Departamento de Matemática Aplicada, Universidad de Granada, 18071 Granada, Spain}
\email{rortega@ugr.es}

\address{$^2$ Departament d'Informàtica, Matemàtica Aplicada i Estadística, Universitat de Girona, 17003 Girona, Spain} \email{david.rojas@udg.edu}



\begin{document}

\begin{abstract}
We prove that Bertrand's property cannot occur in a special-relativistic scenario using the properties of the period function of planar centres. We also explore some integrability properties of the relativistic Coulomb problem and the asymptotic behavior of collision solutions.
\end{abstract}

\maketitle

\section{Introduction}\label{sec:introduccion}

The motion of a particle in the plane under the action of an attractive central force is one of the simplest examples of integrable Hamiltonian system. The energy and angular momentum are two independent first integrals. In two exceptional cases there is a third independent integral and the system becomes super-integrable (see \cite{MosZeh} for more details). As it is well known, these exceptional forces correspond to the harmonic oscillator and Coulomb's law\footnote{In the context of Gravitation theory this is just Newton's law}. The property of super-integrability is related to Bertrand's theorem on central forces. This result says that the exceptional cases mentioned above are the only central forces with the following property: all solutions close to circular motions are periodic. Incidentally we notice that in the case of Coulomb's law any of the two components of the Runge-Lenz vector can be taken as the third integral. This is a classical topic in the study of Kepler problem.

The above discusion is valid within the Newtonian framework but central forces are also meaningful in a special-relativistic scenario (see \cite{BosDamFel2021,BosDamFel2024,KumBha2011,Landau}). Indeed, both relativistic energy and angular momentum are still conserved and a new class of integrable Hamiltonian systems appears. The goal of this paper is to discuss some properties of these systems, particularly those in contrast with the Newtonian case. First we will consider a general attractive central force and we will prove that Bertrand's property cannot occur in the relativistic context. The proof of this result will employ tools from the theory of planar dynamical systems, in particular the properties of the period function of a centre (see \cite{Albouy,OrtRoj2019}). Due to its relevance in electro-dynamics, the case of Coulomb's law is of particular interest and the rest of the paper will be devoted to this case. Explicit formulas for the solutions were obtained in \cite{Landau} using a classical Hamilton-Jacobi approach. In \cite{MunPav2006} it was observed that it is also possible to solve the system using a generalized Runge-Lenz vector. Although this quantity is not conserved, its variation has nice properties and it leads to a strong property of integrability. We will also explore this direction. Finally, we will discuss the asymptotic behaviour of collision solutions. In the Newtonian framework this is a very classical topic leading to Sundman's estimates (see \cite{Sperling}). In particular it is well known that the velocity tends to infinity at the collision. This is impossible in the relativistic world and we will obtain precise estimates of a different nature. The main tool will be a partial regularization somehow inspired by the classical Levi-Civita change of variables.

\section{Statement of the results}

The motion of a relativistic particle of mass $m$ in the presence of a central potential $V$ is described by the Lagrangian
\begin{equation}\label{Lagrangian}
\mathcal L(\textbf{x},\dot{\textbf{x}}) = mc^2- m c^2 \sqrt{1-\frac{|\dot{\textbf{x}}|^2}{c^2}}-V(|\textbf{x}|),\ \ (\textbf{x},\dot {\textbf{x}})\in  (\R^2\setminus\{0\})\times B_c(0),
\end{equation}
where $B_c(0)$ denotes the open ball of radius $c$ centred at the origin. Here $c$ stands for the speed of light and we use the dot notation for time derivation. The additional constant term $mc^2$ may differ from classical references but we added it for the sake of simplicity in the forthcoming sections. 

The Euler-Lagrange equation associated to $\mathcal{L}$ is
\begin{equation}\label{EL}
m\frac{d}{dt}\left(\frac{\dot{\textbf{x}}}{\sqrt{1-\frac{|\dot{\textbf{x}}|^2}{c^2}}}\right) = - V'(|\textbf{x}|)\frac{\textbf{x}}{|\textbf{x}|}.
\end{equation}

Relativistic circular motions are of the type 
\[
\textbf{x}(t)=r_0e^{i\omega t} \text{ with }r_0>0 \text{ and }0<r_0\omega<c.
\] 
We will consider real analytic functions $V:\ ]0,+\infty[\rightarrow\mathbb{R}$ satisfying
\begin{equation}\label{dV-positive}
V'(r)>0 \text{ for each } r\in\ ]0,+\infty[.
\end{equation}
Then it is easy to prove that for each $r_0>0$ there exists a unique circular motion. Let us denote the corresponding frequency by $\omega=\Omega(r_0)$. In the phase space $(\mathbb{R}^2\setminus\{0\})\times B_c(0)$, the orbit associated to the circular motion is
\[
\mathcal{C}_{r_0}:=\{(r_0\xi, i\Omega(r_0)r_0\xi) : \xi\in\mathbb{S}^1\}.
\]
We will say that the potential $V(r)$ has the \emph{Bertrand} property if for every $r_0>0$ there exists a neighborhood $\mathcal{U}_{r_0}\subset\mathbb{C}\times\mathbb{C}$ of $\mathcal{C}_{r_0}$ such that all the solutions of~\eqref{EL} with initial conditions lying in $\mathcal{U}_{r_0}$ are periodic.

\begin{thm}
In the above conditions there is no central potential $V(r)$ with the Bertrand property.
\end{thm}

The proof will be presented in Section~\ref{sec:clairaut}.

Due to the conservation of energy and angular momentum it is not hard to prove that circular motions are orbitally stable in the phase space $(\mathbb{R}^2\setminus\{0\})\times B_c(0)$. However these motions are not Lyapunov stable unless the Bertrand property holds. For this reason the previous result answers in the negative the question posed in~\cite{KumBha2011}.

In the case of Coulomb's law, the potential function takes the form
\[
V(|\textbf{x}|)=-\frac{k}{|\textbf{x}|}, \ k>0.
\]
The rest of the paper is concerned with this particular potential. Taking the position and its conjugated momentum as new variables $(\textbf{q},\textbf{p})\in\Omega:=\R^2\setminus\{0\}\times\R^2$,
\begin{equation}\label{variables}
\textbf{q}=\textbf{x},\ \ \textbf{p}=\frac{\partial \mathcal L}{\partial \dot{\textbf{x}}} = m \frac{\dot{\textbf{x}}}{\sqrt{1-\frac{|\dot{\textbf{x}}|^2}{c^2}}},
\end{equation}
we can deduce from the Lagrangian the expression of the Hamiltonian
\begin{equation}\label{energy}
H(\textbf{q},\textbf{p}) =  \left<\dot{\textbf{q}},\textbf{p}\right>-\mathcal L(\textbf{q},\textbf{p}) = c^2\sqrt{m^2+\frac{|\textbf{p}|^2}{c^2}}-\frac{k}{|\textbf{q}|}-mc^2.
\end{equation}
Here and throughout the paper we use the notation $\left<\cdot,\cdot\right>$ for the scalar product of vectors. The energy $H$ is a constant of motion. A second constant of motion is given by the third component of the angular momentum
\begin{equation}\label{angular_momentum}
L(\textbf{q},\textbf{p}) = \left<\textbf{q},\textbf{Jp}\right>=q_1p_2 - q_2p_1,
\end{equation}
with $\textbf{J}=\left(\begin{smallmatrix}
0 & 1\\
-1 & 0
\end{smallmatrix}\right)$. It is an standard computation to show that $H$ and $L$ are first integrals in involution.

In the non-relativistic scenario, an energy-momentum diagram is used to show the possible values of the constants of motion associated to solutions of the Coulomb's law. The first result is a relativistic version of the previous (see Figure~\ref{fig:diagram}). For any fixed $(\ell,h)\in\R^2$ let us consider the level set 
\[
\mathcal N_{(\ell,h)}=\{ (\mathbf{q},\mathbf{p})\in \Omega : H(\textbf{q},\textbf{p})=h, L(\textbf{q},\textbf{p})=\ell\}.
\]
We also define 
\[
\sigma:=\sqrt{1-\frac{k^2}{\ell^2c^2}}.
\]

\begin{figure}
    \centering
    \includegraphics[scale=1.2]{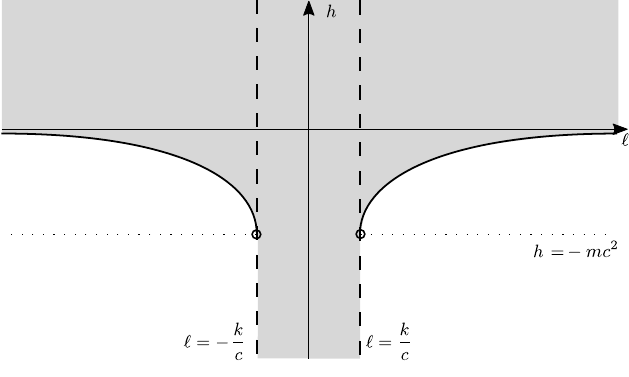}
    \caption{\label{fig:diagram}Energy-Momentum diagram. The shadowed area corresponds to the points $(\ell,h)\in\mathbb{R}^2$ such that $\mathcal N_{(\ell,h)}\neq \emptyset$. The boundary is closed along the curved sections and open along the vertical half-lines starting from (and including) the points $(\pm k/c,-mc^2)$.}
\end{figure}

\begin{prop}\label{integrabilidad}
$\mathcal N_{(\ell,h)}\neq\emptyset$ if and only if 
\[
(\ell,h)\in \{\ell^2c^2-k^2<0 \}\cup \{ \ell^2c^2-k^2\geq 0 \text{ and } h+mc^2(1-\sigma) \geq 0 \}\setminus\{(\pm k/c,-mc^2)\}
\]
Moreover, the points
$\{\ell^2c^2-k^2>0, h+mc^2(1-\sigma)=0 \}$ correspond to circular motions.
\end{prop}

The allowed energy and angular momentum are the values in the shadowed area depicted in Figure~\ref{fig:diagram}, where the system is integrable. Notice that the boundary in the solid line corresponds to circular motion. We refer to Section~\ref{sec:integrability} for more details.

Energy and angular momentum are two first integrals of the system. Let us denote by
\[
\mathscr{E}:=\{(h,\ell)\in\mathbb{R}^2: \ell^2c^2-k^2>0 \text{ and } -mc^2(1-\sigma)<h<0\},
\]
the region on the Energy-Momentum diagram with non-collision bounded motion. Next result shows there are no more continuous independent first integrals in $\mathscr{E}$.

Let us consider the open subset of the phase space
\[
\Omega=\{(\textbf{q},\textbf{p})\in(\mathbb{R}^2\setminus\{0\})\times\mathbb{R}^2 : (H(\textbf{q},\textbf{p}),L(\textbf{q},\textbf{p}))\in\mathscr{E}\}.
\]
We observe that $\Omega$ is invariant under the Hamiltonian flow. Moreover, all solutions $(\textbf{q}(t),\textbf{p}(t))$ lying in this set are globally defined.

A continuous function $F:\Omega\rightarrow\mathbb{R}$ will be called a first integral if
\[
F(\textbf{q}(t),\textbf{p}(t)) = F(\textbf{q}(0),\textbf{p}(0)), \ t\in\mathbb{R},
\]
for all solutions $(\textbf{q}(t),\textbf{p}(t))$ in $\Omega$.

\begin{prop}\label{prop:integralprimera}
Every continuous first integral $F:\Omega\rightarrow\mathbb{R}$ is functionally dependent with $H$ and $L$; that is, there exists a continuous function $\Psi:\mathscr{E}\rightarrow\mathbb{R}$ such that $F=\Psi\circ(H,L)$.
\end{prop}
The previous result is, in fact, valid for any relativistic attractive force with circular motion.

The relativistic Runge-Lenz vector is defined by
\begin{equation}\label{runge}
\textbf{R}(\textbf{q},\textbf{p}) = -mk\gamma(\textbf{p}) \frac{\textbf{q}}{|\textbf{q}|}+\textbf{p}\times (\textbf{q}\times\textbf{p}),
\end{equation}
where 
\begin{equation}\label{gamma}
\gamma(\textbf{p})=\frac{1}{cm}\sqrt{c^2m^2+|\textbf{p}|^2}
\end{equation}
is the Lorentz factor. Here and throughout the paper we use the notation $\textbf{u}\times\textbf{v}$ to denote the cross product of $\textbf{u}$ with $\textbf{v}$. In the previous definition we are abusing the notation of the cross product allowing to act on vectors in $\mathbb{R}^2$. Of course, the cross product is performed with a three-dimensional vector with null third component. Notice that the vector $\textbf{R}$ is indeed contained in the same plane as $\textbf{q}$ and $\textbf{p}$, $\textbf{R}^2\times\{0\}$. Let us consider the mobile reference system $\{\Balpha(t),\Bbeta(t)\}$ given by
\[
\Balpha = \frac{\textbf{q}}{|\textbf{q}|},\ \ \Bbeta = \textbf{J}\Balpha.
\]
Using the argument function of $\Balpha$, $\theta=\theta(t)$, we have $\Balpha = (\cos\theta,\sin\theta)$ and so $\dot{\Balpha}=\dot{\theta}\Bbeta$, $\dot{\Bbeta}=-\dot{\theta}\Balpha$. In this reference system we express the relativistic Runge-Lenz vector as
\[
\textbf{R} = R_{\alpha}\Balpha + R_{\beta}\Bbeta,
\]
with $R_{\alpha} = \left<\textbf{R},\Balpha\right>$ and $R_{\beta} = \left<\textbf{R},\Bbeta\right>$.

\begin{prop}\label{prop:runge}
The relativistic Runge-Lenz vector satisfies the linear differential equation
\begin{equation}\label{edoR}
\begin{cases}
    \frac{d}{d\theta}R_{\alpha}=\sigma^2 R_{\beta},\\
    \frac{d}{d\theta}R_{\beta}=-R_{\alpha},
\end{cases}
\end{equation}
where $\sigma$ is a constant depending on the values of the energy and momentum.
\end{prop}

According to Proposition~\ref{prop:integralprimera} the relativistic Coulomb problem is not super-integrable. However the above result shows that this system enjoys a property of strong integrability that can be described as follows. We are in the presence of an integrable system with two degrees of freedom such that there exists a function $R_{\alpha}$ on the phase space with the properties below,
\begin{itemize}
\item $R_{\alpha}$ is independent with the constants of motion ($\nabla R_{\alpha}$, $\nabla H$ and $\nabla L$ are linearly independent),
\item $R_{\alpha}$ satisfies a linear differential equation of constant coefficients (for each solution $(\textbf{q}(t),\textbf{p}(t))$, the function $t\mapsto R_{\alpha}(\textbf{q}(t),\textbf{p}(t))$ satisfies a linear differential equation. After a change in the independent variable, $\theta=\theta(t)$, this equation has constant coefficients.)
\end{itemize}

This property of strong integrability has been employed in \cite{MunPav2006} to find explicit equations of the orbits. Perhaps it could be of interest to analyze this property and to describe the potentials enjoying it. This refers to both the Newtonian and Relativistic frameworks.

The final result in this paper is an asymptotic description of both radius and argument of the trajectory as the trajectory approach the collision.

\begin{prop}\label{prop:asymptotic}
    Let $(\textbf{q},\textbf{p})\in\mathcal{N}_{(\ell,h)}$ with $\sigma^2<0$ and let $\textbf{q}(t)=r(t)e^{i\theta(t)}$. Set the collision time at $t=0$ and assume $t\geq 0$. Then the following asymptotics hold for $t$ small:
    \begin{equation}
        \begin{cases}
            r(t) = \dfrac{c^2 w_{10}}{k}t+o(t),\\
            \theta(t) = \theta_0 + \dfrac{\ell}{w_{10}}\ln(t)+o(|\ln(t)|),
        \end{cases}
    \end{equation}
    with $\lim_{t\rightarrow 0}\left<\textbf{q}(t),\textbf{p}(t)\right>=w_{10}\neq 0$ and $\theta_0\in\mathbb{R}$.
\end{prop}

\section{Relativistic oscillator and Bertrand's problem}\label{sec:clairaut}

Identifying the plane of motion with $\mathbb{C}$, by writing $\textbf{x}(t)=r(t)e^{i\theta(t)}$, the equations of motion described by the Lagrangian can be written as
\begin{equation}\label{relativistic-motion}
\begin{cases}
m\dfrac{d}{dt}\left(\gamma\dot{r}\right) - \dfrac{\ell^2}{m\gamma r^3} = -V'(r),\\
\ell = m\gamma r^2\dot{\theta},
\end{cases}
\end{equation}
where $\gamma^{-1}=\sqrt{1-\frac{\dot{r}^2+r^2\dot{\theta}^2}{c^2}}$ and $\ell$ is the angular momentum, which is conserved.

In this case the kinetic energy and the angular momentum are constant,
\begin{equation}\label{help}
\gamma = \frac{1}{\sqrt{1-\frac{r_0^2\omega^2}{c^2}}}>1, \ \ell = m\gamma r_0^2\omega>0.
\end{equation}
For each $r_0>0$ the equation~\eqref{relativistic-motion} has a unique circular motion satisfying
\begin{equation}\label{dV-omega}
V'(r_0)=\frac{\ell^2}{m\gamma r_0^3}=m\gamma r_0\omega^2 = \frac{mr_0\omega^2}{\sqrt{1-\frac{r_0^2\omega^2}{c^2}}}.
\end{equation}

The corresponding angular velocity will be denoted by $\omega=\Omega(r_0)$ or, sometimes, $\omega=\Omega(r_0;V)$. In view of~\eqref{help} we also have the function $\ell=L(r_0;V)$ and $\gamma=\Gamma(r_0;V)$.

\begin{lem}\label{lema:momento-parametro}
The function $L(\cdot;V)$ is constant if and only if
\begin{equation}\label{rar}
V'(r)=\frac{ck}{r^2\sqrt{k+c^2m^2r^2}}
\end{equation}
for some constant $k>0$. In this case, $L^2=k$.
\end{lem}
\begin{proof}
First we obtain a formula connecting the functions $L(\cdot;V)$ and $V'$. From~\eqref{help} we deduce that
\[
L=\frac{mr_0^2\Omega}{\sqrt{1-\frac{r_0^2\Omega^2}{c^2}}}.
\]
Since $\Omega$ is positive,
\begin{equation}\label{freq}
\Omega = \frac{cL}{r_0\sqrt{L^2+c^2m^2r_0^2}}.
\end{equation}
Also, from~\eqref{dV-omega},
\begin{equation}\label{simple}
L\Omega = r_0 V'(r_0).
\end{equation}
Eliminating $\Omega$ from the last two identities we obtain
\[
\frac{cL^2}{r_0^2\sqrt{L^2+c^2m^2r_0^2}}=V'(r_0).
\]
Assume now that $L>0$ is independent of $r_0$. We obtain the formula~\eqref{rar} with $L^2=k$. Note that we are using here that there exists a circular motion for every $r_0>0$.

Conversely, assume that $V$ is of the type~\eqref{rar}. Squaring the identity~\eqref{simple},
\[
L^2\Omega^2 = r^2V'(r)^2 = \frac{c^2k^2}{r^2(k+c^2m^2r^2)}.
\]
Together with~\eqref{freq}, for each $r>0$,
\[
\frac{c^2L^4}{r^2(L^2+c^2m^2r^2)} = \frac{c^2k^2}{r^2(k+c^2m^2r^2)}.
\]
From $L^4k+L^4c^2m^2r^2=L^2k^2+k^2c^2m^2r^2$, $\forall r>0$, we deduce that $L^2=k$ and the function $L(\cdot;V)$ is constant.
\end{proof}

From~\eqref{rar} we get the potential
\[
V(r) = -c\sqrt{\frac{k}{r^2}+c^2m^2},
\]
which appeared previously on~\cite{KumBha2011}. 

Following \cite{KumBha2011}, using $\theta=\theta(t)$ as independent variable and the Clairaut's change of variable $\rho = r^{-1}$, equation~\eqref{relativistic-motion} is transformed into
\begin{equation}\label{clairaut}
\frac{d^2\rho}{d\theta^2}+\rho = -\frac{m}{\ell^2}\gamma(\rho,\tfrac{d\rho}{d\theta};\ell)W'(\rho),\ \rho>0
\end{equation}
with $W(\rho)=V(\frac{1}{\rho})$ and $\gamma(\rho,\eta;\ell)=\sqrt{1+\frac{\ell^2}{c^2m^2}(\rho^2+\eta^2)}$. 
On account of the limit case $c\rightarrow +\infty$, equation~\eqref{clairaut} can be interpreted as a relativistic oscillator, although we notice that it is not a typical second order equation with a potential because $\gamma$ depends on $\tfrac{d\rho}{d\theta}$, in contrast with the Euclidean case.

Circular motions of~\eqref{relativistic-motion} are in correspondence with equilibria of~\eqref{clairaut} by $\rho_0=\frac{1}{r_0}$, whereas non-circular periodic solutions of~\eqref{relativistic-motion} correspond to periodic solutions of~\eqref{clairaut} whose minimal period is commensurable with $\pi$. Next result shows that the previous is a necessary and sufficient condition.

\begin{lem}\label{commensurable}
Let $\textbf{x}(t)=r(t)e^{i\theta(t)}$ be a solution of~\eqref{relativistic-motion} and $\rho(\theta)$ be a solution of $\eqref{clairaut}$ for some $\ell\neq 0$. Then $\textbf{x}(t)$ is a non-circular periodic solution if and only if $\rho(\theta)$ is $\Theta$-periodic with $\Theta$ commensurable with $\pi$.
\end{lem}

\begin{proof}
Indeed, given $\rho(\theta)$ a periodic solution of~\eqref{clairaut} with minimal period $\Theta$, we define
\[
t(\theta) = t_0 + \int_0^{\theta} \frac{m\gamma(\rho(\varphi),\frac{d\rho}{d\theta}(\varphi);\ell)}{\ell\rho(\varphi)^2}d\varphi,
\]
The integrand of the previous expression is $\Theta$-periodic, so we can write $t(\theta)=\psi(\theta)+ \sigma\theta$, where $\psi$ is a $\Theta$-periodic function and 
\[
\sigma:=\frac{1}{\Theta}\int_0^{\Theta}\frac{m\gamma(\rho(\varphi),\frac{d\rho}{d\theta}(\varphi);\ell)}{\ell\rho(\varphi)^2}d\varphi.
\]
Consequently, $t(\theta+\Theta) = t(\theta)+\sigma\Theta$. The inverse function $\theta=\theta(t)$ will satisfy $\theta(t+\sigma\Theta)=\theta(t)+\Theta$. Defining $r(t)=\rho(\theta(t))^{-1}$, we notice that
\[
r(t+\sigma\Theta) = \frac{1}{\rho(\theta(t+\sigma\Theta))} = \frac{1}{\rho(\theta(t)+\Theta)} = \frac{1}{\rho(\theta(t))} = r(t).
\]
That is, $r(t)$ is a $\sigma\Theta$-periodic function. Moreover, since $\theta(t+\sigma\Theta)=\theta(t)+\Theta$ then $\theta(t)=\Psi(t) + \frac{t}{\sigma}$, where $\Psi$ is a $\sigma\Theta$-periodic function. Therefore, $\textbf{x}(t)=r(t)e^{i\theta(t)}$ is periodic with minimal period $T=m\sigma\Theta$ if $\Theta=2\pi\frac{n}{m}$ with $\frac{n}{m}$ irreducible. Otherwise the set $\{\textbf{x}(t): t\in\mathbb{R}\}$ is dense in a topological torus and therefore it is not a closed orbit. The previous also works in the converse way.
\end{proof}

In view of Lemma~\ref{commensurable}, Bertrand's property is directly related to prove that the family of second order equations~\eqref{clairaut} is a $\Theta$-isochronous family near equilibria for some $\Theta>0$ commensurable with $\pi$. Notice that $\Theta$ must be independent of the parameter $r_0$. The analyticity of $V$ implies that the local isochronicity must indeed be global. A more precise statement is given by the next result. 

\begin{defi}\label{defi:isocrono}
Assume first that $\ell>0$ is fixed. The system
\begin{equation}\label{system-clairaut}
\dfrac{d\rho}{d\theta} = \eta,\ \dfrac{d\eta}{d\theta} = -\rho-\dfrac{m}{\ell^2}\gamma(\rho,\eta;\ell)W'(\rho)
\end{equation}
has an isochronous center at an equilibrium $(\rho_0,0)$ with $\rho_0>0$ if there exists $\epsilon>0$ such that the solution with initial condition $\rho(0)=\rho_0$, $\eta(0)=\hat\eta$, $0<\hat\eta<\epsilon$ has minimum period $\Theta>0$.

Note that $\Theta$ is independent of $\hat\eta$. An sketched representation of the previous initial conditions is given in Figure~\ref{fig:centro}A. From this picture it is clear the the phase portrait is of the type Figure~\ref{fig:centro}B.

\begin{figure}
\centering
\begin{subfigure}{.45\textwidth}
\centering
\includegraphics[scale=1]{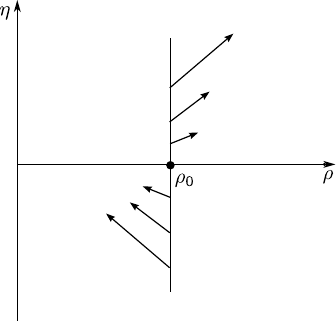} 
\caption{\label{fig:centro_a}}
\end{subfigure}
\begin{subfigure}{.45\textwidth}
\centering
\includegraphics[scale=1]{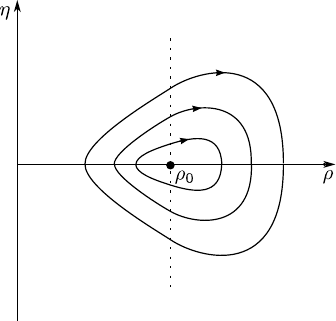}
\caption{\label{fig:centro_b}}
\end{subfigure}
\caption{\label{fig:centro}On the left, the vector field of~\eqref{system-clairaut} on the vertical transversal section from $(\rho_0,0)$ according with Definition~\ref{defi:isocrono}. On the right, the correspondign qualitative phase portrait locally near $(\rho_0,0)$.}
\end{figure}

Moreover, notice that if $\Gamma_n$ is the orbit passing through $(\rho_0,\eta_n)$ with $\eta_n\rightarrow 0$, $\eta_n>0$, then $\Gamma_n$ converges to $\{0\}$ in the Hausdorff distance. This is a consequence of continuous dependence, because the corresponding solution $(\rho_n(\theta),\eta_n(\theta))$ converges to $(0,0)$ uniformly in the interval $[0,\Theta]$.

We will say that the equation~\eqref{system-clairaut} contains a family of isochronous centers if there exists a number $\Theta>0$, an non-empty open interval $]\ell_-,\ell_+[$, and an analytic function $\rho_0(\ell)$ such that for each $\ell$, $\ell_-<\ell<\ell_+$, $(\rho_0(\ell),0)$ is an isochronous center with period $\Theta$.
\end{defi}

\begin{lem}\label{translation}
Assume that the potential $V(r)$ has the Bertrand property. Then the angular momentum function $L(\cdot,V)$ is not constant and the corresponding equation~\eqref{clairaut} contains a family of isochronous centers.
\end{lem}
\begin{proof}
First we prove that the Bertrand property cannot hold if $L(\cdot,V)$ is constant. We know from Lemma~\ref{lema:momento-parametro} that~\eqref{rar} holds with $\ell=\sqrt{k}$. Then
\[
W(\rho)=-c\sqrt{k\rho^2+c^2m^2}
\]
and the system~\eqref{system-clairaut} does not have equilibria unless $\ell=\sqrt{k}$. In this case the set of equilibria is the continuum
\[
\{(\rho_0,0) : \rho_0>0\}.
\]
Since any closed orbit must surround an equilibrium, it is clear the no closed orbit can exist for any $\ell$. In consequence the only periodic solutions of~\eqref{EL} are circular solutions (with positive or negative orientation) and the Bertrand property cannot hold (see Figure~\ref{fig:retrato_fase}).

Assume now that $L(\cdot,V)$ is not constant. Since this function is analytic we can select an interval $]r_-,r_+[$ where it is one-to-one and $L'(r)\neq 0$. For each $r_0\in ]r_-,r_+[$ we consider $\ell=L(r_0,V)$ and obtain the interval $]\ell_-,\ell_+[$. Let $(\rho(\theta,\delta),\eta(\theta,\delta))$ be the solution of~\eqref{system-clairaut} with initial condition $\rho(0)=\rho_0(\ell)$, $\eta(0)=\delta$. From the Bertrand property and Lemma~\ref{commensurable} we know that this solution is periodic with minimal period commensurable with $\pi$. Let this period be denoted by $\Theta=\Theta(\delta,\ell)$. It satisfies
\[
\rho(\Theta,\delta)=\rho_0(\ell).
\]
Since $\frac{\partial\rho}{\partial\theta}(\Theta,\delta)=\eta(\Theta) = \eta(0)=\delta>0$ we can interpret the previous identity as an implicit function problem with unknown $\Theta=\Theta(\delta,\ell)$. This function is analytic and takes values in $\pi\mathbb{Q}$. Therefore $\Theta(\delta,\ell)$ must be locally constant. We have obtained the family of isochronous centers.
\end{proof}

\begin{figure}
\centering
\includegraphics[scale=1]{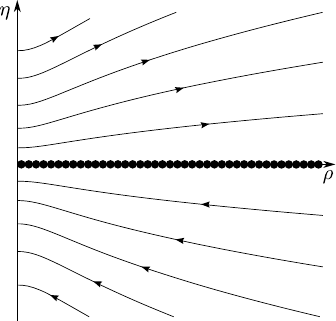}
\caption{\label{fig:retrato_fase}Phase portrait of system~\eqref{system-clairaut} with $\ell=\sqrt{k}$ and $W(\rho)=-c\sqrt{k\rho^2+c^2m^2}$. For the computation we use $m=c=k=1$. Dots represent equilibria.}
\end{figure}

\begin{thm}\label{thm:no-bertrand}
There is no real analytic function $V$ satisfying the Bertrand property.
\end{thm}
\begin{proof}
By means of Lemma~\ref{translation} we prove the result showing that there is no analytic function $V$ different from the type~\eqref{rar} such that~\eqref{system-clairaut} contains a family of isochronous centers.

With the aim of reaching a contradiction, assume that \eqref{system-clairaut} has a family of isochronous centers. That is, there exists a number $\Theta>0$, an non-empty open interval $]\ell_-,\ell_+[$, and an analytic function $\rho_0(\ell)$ such that for each $\ell$, $\ell_-<\ell<\ell_+$, $(\rho_0(\ell),0)$ is an isochronous center of~\eqref{system-clairaut} with period $\Theta$. By Lemma~\ref{lema:momento-parametro}, the function $L(\cdot;V)$ is not constant. Since $L$ is analytic, we can select the interval $]\ell_-,\ell_+[$ where it is one-to-one and consider
$\Lambda:=\{\rho_0(\ell) : \ell_-<\ell<\ell_+\}$. Then, $\Lambda$ is a non-empty interval and for each $\rho_0\in\Lambda$, system~\eqref{system-clairaut} has an isochronous center at $(\rho_0,0)$ with $\ell=\mathscr{L}(\rho_0)$ and $\mathscr{L}(\rho_0):=L\bigl(\frac{1}{\rho_0},V\bigr)$. In particular, $\rho_0$ satisfies the identity
\begin{equation}\label{equilibrium}
\rho_0+\frac{m}{\mathscr{L}(\rho_0)^2}\gamma(\rho_0,0;\mathscr{L}(\rho_0))W'(\rho_0)=0.
\end{equation}
From now on we omit the dependence of $L$ on $\rho_0$ and $V$ for the sake of simplicity.

For a fixed $\ell$, the Jacobian matrix of the vector field defined by~\eqref{system-clairaut} is given by
\[
\left(\begin{matrix}
0 & 1\\
-A(\rho,\eta;\ell) & -B(\rho,\eta;\ell)
\end{matrix}\right)
\]
with
\[
A(\rho,\eta;\ell)=1+\frac{\rho W'(\rho)}{c^2m\gamma(\rho,\eta;\ell)}+\frac{m}{\ell^2}\gamma(\rho,\eta;\ell)W''(\rho),
\]
and
\[
B(\rho,\eta;\ell)=\frac{\eta W'(\rho)}{c^2m\gamma(\rho,\eta;\ell)}.
\]
Since $(\rho_0,0)$ is a center of minimal period $\Theta$ and $B(\rho_0,0;\mathscr{L}(\rho_0))$ is identically zero, then $A(\rho_0,0;\mathscr{L}(\rho_0))=\left(\frac{2\pi}{\Theta}\right)^2$. Invoking the equality~\eqref{equilibrium} we have 
\[
\gamma(\rho_0,0;\mathscr{L}(\rho_0))=-\frac{\rho_0\mathscr{L}(\rho_0)^2}{mW'(\rho_0)}.
\]
After the substitution of the previous in the expression of $A(\rho_0,0;\mathscr{L}(\rho_0)^2)$ and some algebraic manipulations, we find the equality
\begin{equation}\label{recursividad}
W''(\rho_0)=-\frac{a}{\rho_0}W'(\rho_0)-\frac{1}{c^2\mathscr{L}(\rho_0)^2\rho_0}W'(\rho_0)^3,
\end{equation}
with $a:=\left(\frac{2\pi}{\Theta}\right)^2-1>-1$. We note that $a$ does not depend on $\rho_0$.

By implicit derivation of~\eqref{equilibrium} with respect to $\rho_0$ and taking into account the identities \eqref{equilibrium} and~\eqref{recursividad} we can write
\begin{equation}\label{derivada-l}
\mathscr{L}'(\rho_0)=-\frac{(1+a)\mathscr{L}(\rho_0)(c^2m^2+\rho_0^2\mathscr{L}(\rho_0)^2)}{2c^2m^2\rho_0+\rho_0^3\mathscr{L}(\rho_0)^2}.
\end{equation}
Using the previous together with~\eqref{recursividad} and with the aid of an algebraic manipulator, we can compute the derivatives 
\begin{equation}\label{derivada-3}
W'''(\rho_0)=\frac{W'(\rho_0)\bigl(\mu_0(\rho_0)+\mu_1(\rho_0)W'(\rho_0)^2+\mu_2(\rho_0)W'(\rho_0)^4\bigr)}{c^4\rho_0^2\mathscr{L}(\rho_0)^4(2c^2m^2+\rho_0^2\mathscr{L}(\rho_0)^2)}
\end{equation}
and
\begin{equation}\label{derivada-4}
W^{(4)}(\rho_0)=\frac{-W'(\rho_0)\bigl(\eta_0(\rho_0)+\eta_1(\rho_0)W'(\rho_0)^2+\eta_2(\rho_0)W'(\rho_0)^4+\eta_3(\rho_0)W'(\rho_0)^6 \bigr)}{c^6\rho_0^3\mathscr{L}(\rho_0)^6(2c^2m^2+\rho_0^2\mathscr{L}(\rho_0)^2)^3}
\end{equation}
with
\begin{align*}
\mu_0(\rho_0)&=a(1+a)c^4\mathscr{L}(\rho_0)^4(2c^2m^2+\rho_0^2\mathscr{L}(\rho_0)^2),\\
\mu_1(\rho_0)&=c^2\mathscr{L}(\rho_0)^2(6ac^2m^2+(-1+2a)\rho_0^2\mathscr{L}(\rho_0)^2),\\
\mu_2(\rho_0)&=6c^2m^2+3\rho_0^2\mathscr{L}(\rho_0)^2,\\
\eta_0(\rho_0)&=a(1+a)(2+a)c^6\mathscr{L}(\rho_0)^6(2c^2m^2+\rho_0^2\mathscr{L}(\rho_0)^2)^3,\\
\eta_1(\rho_0)&=c^4\mathscr{L}(\rho_0)^4(8a(4+7a)c^6m^6+12a(2+5a)c^4m^4\rho_0^2\mathscr{L}(\rho_0)^2\\
&\phantom{=}+2(-1+10a^2)c^2m^2\rho_0^4\mathscr{L}(\rho_0)^4+3a^2\rho_0^6\mathscr{L}(\rho_0)^6),\\
\eta_2(\rho_0)&=9c^2\mathscr{L}(\rho_0)^2(4ac^2m^2+(-1+a)\rho_0^2\mathscr{L}(\rho_0)^2)(2c^2m^2+\rho_0^2\mathscr{L}(\rho_0)^2)^2,\\
\eta_3(\rho_0)&=15(2c^2m^2+\rho_0^2\mathscr{L}(\rho_0)^2)^3.
\end{align*}

Now we consider polar coordinates $(R,\phi)$ on~\eqref{system-clairaut} centred at the equilibrium $(\rho,\eta)=(\rho_0,0)$ for a fixed $\ell=\mathscr{L}(\rho_0)$. That is, $\rho=\rho_0+R\cos(\phi)$, $\eta=R\sin(\phi)$. We obtain
\begin{equation}\label{sistema-polares}
\begin{cases}\vspace*{.2cm}
\dfrac{dR}{d\theta}&=-\sin(\phi)F(R,\phi,\rho_0),\\ 
\dfrac{d\phi}{d\theta}&=-1-\dfrac{\cos(\phi)}{R}F(R,\phi,\rho_0).
\end{cases}
\end{equation}
with
\[
F(R,\phi,\rho_0):=\rho_0+\frac{m}{\mathscr{L}(\rho_0)^2}\sqrt{1+\frac{(R^2+\rho_0^2+2R\rho_0\cos(\phi))\mathscr{L}(\rho_0)^2}{c^2m^2}}W'(\rho_0+R\cos(\phi)).
\]

For small $\xi>0$ let $R(\phi,\xi,\rho_0)$ be the solution with initial condition $R(0,\xi,\rho_0)=\xi$ of the differential equation
\[
\frac{dR}{d\phi}=\frac{\sin(\phi)F(R,\phi,\rho_0)}{1+\frac{\cos(\phi)}{R}F(R,\phi,\rho_0)}.
\]
Using the second equation in~\eqref{sistema-polares} we have that the period function of the center at the origin of~\eqref{sistema-polares} can be written as
\[
P(\xi)=\int_0^{2\pi}\frac{1}{1+\dfrac{F(R(\phi,\xi,\rho_0),\phi,\rho_0)}{R(\phi,\xi,\rho_0)}\cos(\phi)}d\phi.
\]
We compute the first terms on the asymptotic expansion of $P(\xi)$ near $\xi=0$, the so-called period constants. To do so, we first compute the first terms on the asymptotic expansion of $R(\phi,\xi,\rho_0)$ near $\xi=0$ using the equality
\[
\frac{dR}{d\phi}(\phi,\xi,\rho_0)\left(1+\frac{\cos(\phi)}{R(\phi,\xi,\rho_0)}F(R(\phi,\xi,\rho_0),\phi,\rho_0)\right)=\sin(\phi)F(R(\phi,\xi,\rho_0),\phi,\rho_0).
\]
Taking $R(\phi,\xi,\rho_0)=R_1(\phi,\rho_0)\xi+R_2(\phi,\rho_0)\xi^2+O(\xi^3)$ and on account of the expressions \eqref{recursividad}-\eqref{derivada-4}, with the help of an algebraic manipulator we can compute
\[
R_1(\phi,\rho_0)=\frac{\sqrt{1+a}}{\sqrt{1+a\cos^2(\phi)}}
\]
and
\[
R_2(\phi,\rho_0)=\frac{R_1(\phi,\rho_0)\left(\lambda_1(\rho_0)+\frac{R_1(\phi,\rho_0)}{1+a\cos(\phi)^2}(\lambda_2(\rho_0)\cos(\phi)+\lambda_3(\rho_0)\cos(\phi)^3)\right)}{6\rho_0(2c^4m^4+3c^2m^2\rho_0^2\mathscr{L}(\rho_0)^2+\rho_0^4\mathscr{L}(\rho_0)^4)},
\]
with
\begin{align*}
\lambda_1(\rho_0)&=-(2ac^4m^4+3(2+a)c^2m^2\rho_0^2\mathscr{L}(\rho_0)^2+(2+a)\rho_0^4\mathscr{L}(\rho_0)^4),\\
\lambda_2(\rho_0)&=3\rho_0^2\mathscr{L}(\rho_0)^2(2c^2m^2+\rho_0^2\mathscr{L}(\rho_0)^2),\\
\lambda_3(\rho_0)&=2a(1+a)c^4m^4+3a(3+a)c^2m^2\rho_0^2\mathscr{L}(\rho_0)^2+(-1+3a+a^2)\rho_0^4\mathscr{L}(\rho_0)^4.
\end{align*}
We can now compute the first terms in the asymptotic expansion of $P(\xi)$ near $\xi=0$, obtaining
\[
P(\xi)=\Theta-\frac{\pi c^{10}m^{10} Q\left(\frac{\rho_0^2\mathscr{L}(\rho_0)^2}{c^2m^2}\right)}{12\sqrt{1+a}\rho_0^2(2c^2m^2+\rho_0^2\mathscr{L}(\rho_0)^2)^3(c^2m^2+\rho_0^2\mathscr{L}(\rho_0)^2)^2}\xi^2+\cdots
\]
with $Q(x)=8(3-a)a+28(3-a)ax+2(3-a)(8+19a)x^2+(63+46a-25a^2)x^3+(22+10a-8a^2)x^4+(2+2a-a^2)x^5$. Since we are assuming that $\Theta$ is a constant independent of $\rho_0$, also $a$ has this property. In consequence, the polynomial $Q(x)$ is independent of $\rho_0$.

If $(\rho_0,0)$ were an isochronous center, the function $P(\xi)$ should be constant. This should imply that the function
\[
\sigma(\rho_0):=\frac{\rho_0^2\mathscr{L}(\rho_0)^2}{c^2m^2}
\]
will take values in the set of roots of $Q(x)$. This polynomial has degree $5$ if $a\neq 1\pm\sqrt{3}$ and degree $4$ if $a=1\pm\sqrt{3}$. In any case $Q(x)$ has at most five roots. Now we can deduce that the continuous function $\sigma(\rho_0)$ is indeed a constant independent of $\rho_0$. Assume
\begin{equation}\label{trick}
\sigma(\rho_0)=K(a)
\end{equation}
where $K(a)>0$ is a positive root of $Q(x)$. The rest of the proof aims to show that this identity cannot hold for an isochronous center. Indeed, assuming that~\eqref{trick} holds, by equality~\eqref{equilibrium} we have
\[
W'(\rho_0)=-\frac{mc^2K(a)}{\sqrt{1+K(a)}\rho_0}.
\]
Invoking the differential equation~\eqref{recursividad}, we find
\[
K(a)^2+(a-1)(1+K(a))K(a)=0,
\]
which produces $K(a)=\frac{1-a}{a}$. Evaluating the polynomial,
\[
Q\bigl(\tfrac{1-a}{a}\bigr)=\frac{2(1+a)(1+6a-9a^2+6a^3)}{a^5}.
\]
Therefore, $K(a)=\frac{1-a}{a}$ is a positive root of $Q(x)$ if and only if $0<a<1$ and $1+6a-9a^2+6a^3=0$. Since the polynomial $1+6a-9a^2+6a^3$ has no roots on the interval $[0,1]$, we have reached a contradiction. This finishes the proof of the result.
\end{proof}

\begin{rem}
Although the function $\mathscr{L}(\rho_0)$ is implicitly defined by~\eqref{equilibrium}, it can be explicitly defined since it solves an algebraic equation of low degree. Indeed, from the formula expressing the angular momentum of circular motions, we have that the function $\tilde{\Gamma}(\rho_0):=\Gamma\bigl(\tfrac{1}{\rho_0};V\bigr)$ is given by
\[
\tilde{\Gamma}(\rho_0)=\sqrt{1+\frac{\rho_0^2\mathscr{L}(\rho_0)^2}{c^2m^2}}\geq 1.
\] 
Therefore, using~\eqref{equilibrium}, $\tilde{\Gamma}$ satisfies the equation
\[
\gamma^2+\frac{1}{c^2m^2}\rho_0W'(\rho_0)\gamma -1 =0.
\]
Solving the equation we find
\[
\tilde{\Gamma}(\rho_0) = \frac{1}{2}\left(-\frac{1}{mc^2}\rho_0W'(\rho_0)+\sqrt{\frac{1}{m^2c^4}\rho_0^2W'(\rho_0)^2+4}\right),
\]
where we choose the positive determination of the square root in order to verify $\gamma\geq 1$. With this equation~\eqref{recursividad} writes
\[
W''(\rho_0)=-\frac{a}{\rho_0}W'(\rho_0)-\frac{\rho_0}{m^2c^4(\tilde{\Gamma}(\rho_0)^2-1)}W'(\rho_0)^3.
\]
The previous equation can be integrated to obtain an implicit description of the possible functions $W$ that could be (but are not) candidate to give a positive answer to Bertrand's question. This gives an alternative way of proving Theorem~\ref{thm:no-bertrand}.

We finally note that, using the above identities, $c^2\tilde{\Gamma}(\rho_0)\rightarrow+\infty$ as $c\rightarrow+\infty$. The same holds for $c^4(\tilde{\Gamma}(\rho_0)^2-1)$ and the equation~\eqref{recursividad} tends to
\[
W''(\rho_0)=-\frac{a}{\rho_0}W'(\rho_0).
\]

The latest equality appeared previously in \cite{OrtRoj2019} as the differential equation that Bertrand's potential candidates must satisfy in a Newtonian world.
\end{rem}

\section{Integrability of the special-relativistic Coulomb's law} \label{sec:integrability}

This section is devoted to prove Proposition~\ref{integrabilidad}. Let us consider the Lagrangian 
\begin{equation}\label{Coulomb}
\mathcal L(\textbf{x},\dot{\textbf{x}}) = mc^2- m c^2 \sqrt{1-\frac{|\dot{\textbf{x}}|^2}{c^2}}+\frac{k}{|\textbf{x}|},\ \ (\textbf{x},\dot {\textbf{x}})\in  (\R^2\setminus\{0\})\times B_c(0),
\end{equation}
associated to the relativistic Coulomb problem. As shown in Section 2, the energy $H$ in \eqref{energy} and the angular momentum $L$ in \eqref{angular_momentum} are two first integrals of motion. It is a computation to show that
 \[
 \nabla H = \begin{pmatrix} \frac{k}{|\textbf{q}|^3} \textbf{q}\\ \frac{1}{\sqrt{m^2+\frac{|\textbf{p}|^2}{c^2}}} \textbf{p} \end{pmatrix} \ \text{and}\  \ \nabla L = \begin{pmatrix} \textbf{Jp} \\ -\textbf{Jq} \end{pmatrix},
 \]
 and so
 \[
 \{H,L\} = \left<\nabla H, \begin{pmatrix} -\textbf{Jq} \\ \textbf{Jp} \end{pmatrix}\right> = -\frac{k}{|\textbf{q}|^3} \left<\textbf{q}, \textbf{Jq}\right> + \frac{1}{\sqrt{m^2+\frac{|\textbf{p}|^2}{c^2}}} \left<\textbf{p},\textbf{Jp}\right> = 0.
 \]
 That is, $H$ and $L$ are first integrals in involution. Let us study their linear independence. Since $\textbf{q}\neq\textbf{0}$, both vectors $\nabla H$ and $\nabla L$ are non zero in $\R^4$. Therefore $\nabla H$ and $\nabla L$ are linearly dependent if and only if there exists $\lambda\neq 0$ such that $\lambda \nabla H = \nabla L$. From this equality we obtain the system of equations
\[
\begin{cases}
   \dfrac{\lambda k}{|{\textbf{q}}|^3}\textbf{q} & =  \textbf{Jp}\\
   \dfrac{\lambda}{\sqrt{m^2+\frac{|\textbf{p}|^2}{c^2}}} \textbf{p} & = -\textbf{Jq}
\end{cases}.
\]
Using $\textbf{J}^2=-\textbf{I}$, 
\[
\textbf{q} = -\textbf{J}^2\textbf{q} = \dfrac{\lambda}{\sqrt{m^2+\frac{|\textbf{p}|^2}{c^2}}} \textbf{Jp} = \dfrac{1}{\sqrt{m^2+\frac{|\textbf{p}|^2}{c^2}}} \dfrac{\lambda^2 k}{|\textbf{q}|^3}\textbf{q}.
\]
Thus,
\[
\lambda^2 = \frac{|\textbf{q}|^3}{k}\sqrt{m^2+\frac{|\textbf{p}|^2}{c^2}}.
\]
From this last equality we obtain that the vectors $\nabla H$ and $\nabla L$ are linearly dependent in $\R^4$ if and only if
\[
\textbf{Jp} = \pm \frac{k^{\frac{1}{2}}}{|\textbf{q}|^{\frac{3}{2}}}\left(m^2+\frac{|\textbf{p}|^2}{c^2}\right)^{\frac{1}{4}}\textbf{q}.
\]
The subset 
\[
\mathcal C = \left\{(\textbf{q},\textbf{p})\in \Omega : \textbf{Jp} = \pm \frac{k^{\frac{1}{2}}}{|\textbf{q}|^{\frac{3}{2}}}\left(m^2+\frac{|\textbf{p}|^2}{c^2}\right)^{\frac{1}{4}}\textbf{q}\right\}
\]
is closed in $\Omega$ and, therefore, the system is integrable on the open subset $\Omega^*=\Omega\setminus\mathcal C$. Moreover, we notice that the points in $\mathcal C$ satisfy $q\perp p$ and 
\[
|\textbf{q}||\textbf{p}|^2 = k\sqrt{m^2+\frac{|\textbf{p}|^2}{c^2}}.
\]
In particular, the set $\mathcal{C}$ is formed by circular solutions.

\begin{proof}[Proof of Proposition~\ref{integrabilidad}]
Notice that $|\ell|=|L(\textbf{q},\textbf{p})|\leq |\textbf{p}||\textbf{q}|$ and the equality holds if and only if $p\perp q$. Let us prove the sufficient implication assuming $\mathcal N_{(\ell,h)}\neq \emptyset$ and taking $(\textbf{q},\textbf{p})\in\mathcal N_{(\ell,h)}.$ Assume $\ell\neq 0$, using the identity $H(\textbf{q},\textbf{p})=h$ we have
\[
h\geq c^2\sqrt{m^2+\frac{|\textbf{p}|^2}{c^2}}-\frac{k}{|\ell|}|\textbf{p}|-mc^2 \geq \mu(\ell)
\]
where $\mu(\ell)=\inf_{[0,\infty[}\psi_{\ell}(x)$ and 
\[
\psi_{\ell}(x) = c^2\sqrt{m^2+\frac{x^2}{c^2}}-\frac{k}{|\ell|}x-mc^2.
\]
The infimum $\mu(\ell)$ is finite when $c\geq \frac{k}{|\ell|}$ and a straight computation shows that $\mu(\ell)=mc^2(\sigma-1)$ in that case. Consequently, from $h\geq \mu(\ell)$ we have $h+mc^2(1-\sigma)\geq 0$. Moreover, the infimum is achieved if and only if $|\ell|>\frac{k}{c}$. 

The boundary case when the infimum is achieved corresponds precisely to the points of $\mathcal C$. Indeed, the infimum of $\psi_{\ell}(|\textbf{p}|)$ is taken when 
\[
|\ell||\textbf{p}| = k\sqrt{m^2+\frac{|\textbf{p}|^2}{c^2}},
\]
which is satisfied by the points of $\mathcal C$ since $|\ell|=|\textbf{q}||\textbf{p}|$.

Let us show the necessity of the condition. Let us take a point $(\textbf{q},\textbf{p})\in\Omega$, $\textbf{q}=\lambda\textbf{e}_1$, $\textbf{p}=\mu \textbf{e}_2$, $\lambda>0$, $\mu\in\mathbb{R}$. Therefore $L(\textbf{q},\textbf{p})=\lambda\mu=\ell$ and $H(\textbf{q},\textbf{p})=\psi_{\ell}(|\mu|)=h$. If $|\ell|\neq 0$, choosing $(\ell,h)$ accordingly to the statement implies that a positive value $|\mu|$ can be chosen such that the equality holds. Once $\mu$ is fixed, we obtain $\lambda=\ell/\mu$. The case $|\ell|=0$ follows from taking $\textbf{q}=\lambda\textbf{e}_1$ and $\textbf{p}=\mu\textbf{e}_1$.
\end{proof}

\section{Dependence of a third constant of motion}\label{sec:3integral}
Let us consider $(h,\ell)\in\mathscr{E}$. It is a well-known result that solutions for these constants of motion are periodic or quasi-periodic. As we showed in Lemma~\ref{commensurable} periodicity happens when $\Theta$ is commensurable with $\pi$, whereas quasi-periodic solutions appear in the other case. Following Landau~\cite{Landau} or Section~\ref{sec:clairaut} with $V(r)=\frac{k}{r}$, one find for the relativistic Coulomb problem
\[
\Theta = \sqrt{1-\frac{k^2}{\ell^2c^2}}.
\]
In particular, $\Theta$ depends on the angular momentum $\ell$.

Let us fix $(\ell,h)\in\mathscr{E}$ such that $\Theta\notin\pi\mathbb{Q}$. By Lemma~\ref{commensurable} the solution $\textbf{x}(t)$ is dense in the torus described by the first integrals $H=h$ and $L=\ell$. In consequence, $F$ is constant in the set $\mathcal{N}_{(\ell,h)}$. Since the values of $\ell$ such that $\Theta\notin\pi\mathbb{Q}$ are dense in $\mathscr{E}$, we have that $F$ is constant in each $(h,\ell)\in\mathscr{E}$. Consequently, $F$ is a continuous function of $(h,\ell)$.

\section{The relativistic Runge-Lenz vector}\label{sec:runge}

In this section we describe the relativistic Runge-Lenz vector with the aim of proving Proposition~\ref{prop:runge}. Similar computations can be found in \cite{MunPav2006} for a Hamilton-like vector. From the definition of the Lorentz factor in \eqref{gamma}, notice that $\gamma(\textbf{p})\geq 1$ for any $\textbf{p}\in\mathbb{R}^2$. Invoking the identity $\left<\textbf{u}\times \textbf{v},\textbf{w}\right>=\left<\textbf{u},\textbf{v}\times \textbf{w}\right>$ together with the definition of the Runge-Lenz vector $\textbf{R}$ in \eqref{runge} we obtain
\[
\left<\textbf{R},\textbf{q}\right> = -mk\gamma|\textbf{q}|+|\textbf{q}\times\textbf{p}|^2,
\]
which implies the relativistic conic equation
\begin{equation}\label{cr}
    |\textbf{q}|+\frac{1}{mk\gamma}\left<\textbf{R},\textbf{q}\right> = \frac{1}{mk\gamma}|\textbf{q}\times\textbf{p}|^2,
\end{equation}
for any $(\textbf{q},\textbf{p})\in(\mathbb{R}^2\setminus\{0\})\times\mathbb{R}^2$. Let us consider the Hamiltonian system
\begin{equation}\label{hamiltonian-kepler}
\frac{d\textbf{q}}{dt}=\frac{\textbf{p}}{m\gamma} , \ \ \frac{d\textbf{p}}{dt} = -k \frac{\textbf{q}}{|\textbf{q}|^3}.
\end{equation}
Unlike the Newtonian case, the relativistic Runge-Lenz vector is not a first integral of motion. However an explicit expression for $\mathbf{R}$ can be obtained and from it we can deduce the expression of the orbits. 

We compute the variation of the relativistic Runge-Lenz vector using
\begin{align*}
\frac{d}{dt}\left(\frac{\textbf{q}}{|\textbf{q}|}\right)&=\frac{\dot{\textbf{q}}}{|\textbf{q}|}-\frac{\textbf{q}}{|\textbf{q}|^3}\left<\textbf{q},\dot{\textbf{q}}\right>=\frac{1}{|\textbf{q}|^3}(\textbf{q}\times\dot{\textbf{q}})\times\textbf{q}\\
&=\frac{1}{m\gamma}(\textbf{q}\times\textbf{p})\times \frac{\textbf{q}}{|\textbf{q}|^3}=-\frac{1}{mk\gamma}(\textbf{q}\times\textbf{p})\times \frac{d\textbf{p}}{dt}.
\end{align*}
Since $\ell = |\textbf{q}\times\textbf{p}|$ is first integral, $\Bell = \ell \textbf{e}_3$. Due to the reversibility of the system, it is not restrictive to consider the curve $(\textbf{q}(t),\textbf{p}(t))$ being positive oriented. That is, $\textbf{q}\times\textbf{p}=\ell \textbf{e}_3$ with $\ell\geq 0$. Thus,
\[
\frac{d}{dt}\left(\frac{k\textbf{q}}{|\textbf{q}|}\right) = -\frac{1}{m\gamma} \Bell\times\frac{d\textbf{p}}{dt}=-\frac{1}{m\gamma}\frac{d}{dt}(\Bell\times\textbf{p}).
\]
And so we deduce
\begin{equation}\label{vR}
    \frac{d\textbf{R}}{dt} = - mk \frac{d\gamma}{dt} \frac{\textbf{q}}{|\textbf{q}|}.
\end{equation}
The vector $\textbf{R}$ is confined in the plane $\mathbb{R}^2\times\{0\}$. We consider the mobile reference system $\{\Balpha(t),\Bbeta(t)\}$ given by
\[
\Balpha := \frac{\textbf{q}}{|\textbf{q}|},\ \ \Bbeta := \textbf{J}\Balpha.
\]
We can write $\Balpha = (\cos\theta,\sin\theta)$, where $\theta=\theta(t)$ is the argument function of $\Balpha$. In this reference system the relativistic Runge-Lenz vector is written as
\[
\textbf{R} = R_{\alpha}\Balpha + R_{\beta}\Bbeta,
\]
with $R_{\alpha} = \left<\textbf{R},\Balpha\right>$ and $R_{\beta} = \left<\textbf{R},\Bbeta\right>$. Direct time derivation gives $\dot{\Balpha}=\dot{\theta}\Bbeta$ and $\dot{\Bbeta}=-\dot{\theta}\Balpha$.
Taking time derivatives on the previous equality and using \eqref{vR} we obtain
\[
(\dot{R}_{\alpha}-\dot{\theta}R_{\beta})\Balpha + (R_{\alpha}\dot{\theta}+\dot{R}_{\beta})\Bbeta = -mk\frac{d\gamma}{dt}\Balpha,
\]
and we arrive to the system
\[
\begin{cases}
    \dot{R}_{\alpha}-\dot{\theta}R_{\beta} = -mk\dfrac{d\gamma}{dt},\\
    \dot{R}_{\beta}+ \dot{\theta}R_{\alpha} = 0.
\end{cases}
\]
From now on let us assume $\ell>0$ so that $\theta=\theta(t)$ is a diffeomorphism between certain intervals. Taking $\theta$ as independent variable, $\textbf{R}=\textbf{R}(\theta)$ and denoting with $'$ the differentiation with respect to $\theta$, we obtain
\begin{equation}\label{sR}
\begin{cases}
    R_{\alpha}'-R_{\beta} = -mk\gamma'(\theta),\\
    R_{\beta}'+R_{\alpha}=0.
\end{cases}
\end{equation}
From the definition of $\textbf{R}$,
\[
R_{\alpha} = \left<\textbf{R},\Balpha\right>=-mk\gamma + \left<\textbf{p}\wedge(\textbf{q}\wedge\textbf{p}),\frac{\textbf{q}}{|\textbf{q}|}\right>=
-mk\gamma + \frac{1}{|\textbf{q}|}|\textbf{q}\wedge\textbf{p}|^2 = -mk\gamma + \frac{\ell^2}{|\textbf{q}|},
\]
and from the conservation of energy $h=(\gamma-1)mc^2-\frac{k}{|\textbf{q}|}$, so we can bind $R_{\alpha}$ and $\gamma$ with the expression
\begin{equation}\label{cRG}
    R_{\alpha}=-\frac{\ell^2}{k}(h+mc^2) + \frac{m}{k}(\ell^2c^2-k^2)\gamma.
\end{equation}
From the previous equality we obtain
\[
\gamma ' = \frac{k}{m(\ell^2c^2-k^2)} R'_{\alpha},
\]
and substituting in~\eqref{sR} we arrive to the linear system
\begin{equation}
\begin{cases}
    R'_{\alpha}=\sigma^2 R_{\beta},\\
    R'_{\beta}=-R_{\alpha}.
\end{cases}
\end{equation}
where $\sigma^2=1-\dfrac{k^2}{\ell^2c^2}$. The previous discussion proves Proposition~\ref{prop:runge}.

\section{A description of the motion at the collision}
In this section we prove Proposition~\ref{prop:asymptotic}. The motions of the relativistic Coulomb problem for angular momentum $|\ell|<k/c$ collide with the singularity at the origin as shown, for instance, in~\cite{Landau}. Let us set the instant of collision at $t=0$ and approaching it from $t>0$. Let us consider the map
\[
\textbf{w}:(\mathbb{R}^2\setminus\{0\})\times\mathbb{R}^2 \rightarrow \mathbb{R}^2
\]
defined by $\textbf{w}(\textbf{q},\textbf{p})=(\left<\textbf{q},\textbf{p}\right>,\left<\textbf{q}\wedge\textbf{p},\textbf{e}_3\right>)$ and let us denote by $w_1$ and $w_2$ the first and second component of $\textbf{w}$, respectively. We notice that the map $\textbf{w}$ is polynomial and the map 
\[
(\textbf{q},\textbf{p})\in(\mathbb{R}^2\setminus\{0\}\times\mathbb{R}^2) \mapsto (\textbf{q},\textbf{w})\in (\mathbb{R}^2\setminus\{0\}\times\mathbb{R}^2)
\]
is a diffeomorphism with inverse given by
\[
\textbf{p}= w_1\frac{\textbf{q}}{|\textbf{q}|}+w_2\frac{\textbf{J}\textbf{q}}{|\textbf{q}|}.
\]
It is an straightforward computation to verify the identity
\[
|\textbf{w}|^2=|\textbf{q}|^2|\textbf{p}|^2.
\]
Using the previous on the definition of $\gamma$ and the energy, we obtain the identities
\[
\gamma = \frac{1}{|\textbf{q}|^2}\sqrt{m^2|\textbf{q}|^2+\frac{|\textbf{w}|^2}{c^2}},
\]
and
\[
(h+mc^2)|\textbf{q}|=c^2\sqrt{m^2|\textbf{q}|^2+\frac{|\textbf{w}|^2}{c^2}}-k.
\]
The set
\[
\mathcal{M}:=\{(\textbf{q},\textbf{w},h)\in \mathbb{R}^2\times\mathbb{R}^2\times \mathbb{R} : (h+mc^2)^2|\textbf{q}|^2=\bigl[c^2\sqrt{m^2|\textbf{q}|^2+\tfrac{|\textbf{w}|^2}{c^2}}-k\bigr]^2\}
\]
is a smooth manifold of dimension four that is invariant under the flow
\begin{equation}\label{flow}
\begin{cases}
    \dfrac{d\textbf{q}}{dt} = \frac{1}{\sqrt{m^2|\textbf{q}|^2+\frac{|\textbf{w}|^2}{c^2}}}\left(w_1\frac{\textbf{q}}{|\textbf{q}|}+w_2\frac{J\textbf{q}}{|\textbf{q}|}\right),\\
    \dfrac{dw_1}{dt}=h+mc^2-\frac{m^2c^2|\textbf{q}|}{\sqrt{m^2|\textbf{q}|^2+\frac{|\textbf{w}|^2}{c^2}}},\\
    \dfrac{dw_2}{dt}=\dfrac{dh}{dt}=0.
\end{cases}
\end{equation}
This vector field is discontinuous at $\textbf{q}=0$ but it is bounded. Indeed, letting $\textbf{q}=re^{i\theta}$,
\begin{equation}\label{flow-polar}
\begin{cases}
    \dfrac{dr}{dt}=\dfrac{w_1}{\sqrt{m^2r^2+\frac{|\textbf{w}|^2}{c^2}}},\\
    \dfrac{d\theta}{dt}=\dfrac{\omega_2}{r\sqrt{m^2r^2+\frac{|\textbf{w}|^2}{c^2}}}.
\end{cases}
\end{equation}
Notice that the first equation is smooth at $r=0$, whereas the second equation corresponding to the argument is singular. Assume $r(t)\rightarrow 0$ as $t\rightarrow 0$ with $t>0$. Then, from the second equation of~\eqref{flow},
\[
\frac{d w_1}{dt}=h+mc^2+o(1),
\]
implying
\[
w_1(t)=w_{10}+(h+mc^2)t+o(t).
\]
Assume that $w_{10}\neq 0$. Using the previous equality in the first equation of~\eqref{flow-polar} we obtain
\[
\frac{dr}{dt}=\frac{c^2w_{10}}{k}+o(1),
\]
and so
\[
r(t)=\frac{c^2w_{10}}{k}t+o(t).
\]
Consequently,
\[
\frac{1}{r(t)}=\frac{k}{c^2w_{10}t}+o(1/t).
\]
Using the previous equality in the second equation of~\eqref{flow-polar} we obtain
\[
\frac{d\theta}{dt}=\frac{w_2}{w_{10}t}+o(1/t)
\]
and, therefore,
\[
\theta(t)=\theta_0+\frac{w_2}{w_{10}}\ln(t)+o(|\ln(t)|).
\]
Observe that the condition $w_{10}\neq 0$ is not restrictive. Indeed, if $w_{10}=0$ then $w_1(t)=\left<\textbf{q}(t),\textbf{p}(t)\right>\rightarrow 0$ as $t\rightarrow 0$. In particular, $|\textbf{w}(t)|\rightarrow \ell$ as $t\rightarrow 0$. Besides, from the conservation of energy, $|\textbf{w}(t)|\rightarrow \frac{k}{c}$ as $t\rightarrow 0$ so we have $|\ell| = \frac{k}{c}$ at the collision, which contradicts $|\ell|<k/c$.

\section*{Acknowledgements}
This work was financially supported by the Ministerio de Ciencia, Innovación y Universidades / Agencia Estatal de Investigación grant numbers PID2021-128418NA-I00 and PID2023-146424NB-I00; and by the Generalitat de Catalunya grant 2021 SGR 00113. D.R. is a Serra Húnter Fellow.

\end{document}